\documentclass{amsart}
\usepackage{tikz}
\usepackage{dsfont}
\usetikzlibrary{matrix}
\usepackage{tikz-cd}
\usepackage{amsmath}
\usepackage{amsfonts}
\usepackage{cases}
\usepackage{mathrsfs}
\usepackage{bbm}
\usepackage[bookmarks,   colorlinks]{hyperref}
\usepackage[noabbrev,capitalize]{cleveref}
\usepackage{amssymb}
\usepackage{txfonts}
\usepackage{amscd}
\usepackage{amsfonts,    latexsym,   amsmath,   amsthm,   amsxtra,   mathdots}
\usepackage { graphicx } 
\usepackage{eufrak}
\usepackage[all,   cmtip]{xy}
\RequirePackage{amsmath} 
\RequirePackage{amssymb}
\usepackage{color}
\usepackage{colordvi}
\usepackage{multicol}
\usepackage{tikz}
\usepackage{tikz-cd}

\usepackage{graphicx}
\usepackage{amsmath}
\usepackage{amsmath,   amscd}
\usetikzlibrary{matrix,   arrows}
\usepackage{textcomp}
\newcommand{\ie}{\textit{i}.\textit{e}.}
\makeatletter
\DeclareFontFamily{OMX}{MnSymbolE}{}
\DeclareSymbolFont{MnLargeSymbols}{OMX}{MnSymbolE}{m}{n}
\SetSymbolFont{MnLargeSymbols}{bold}{OMX}{MnSymbolE}{b}{n}
\DeclareFontShape{OMX}{MnSymbolE}{m}{n}{
	<-6>  MnSymbolE5
	<6-7> MnSymbolE6
	<7-8>  MnSymbolE7
	<8-9>  MnSymbolE8
	<9-10> MnSymbolE9
	<10-12> MnSymbolE10
	<12->   MnSymbolE12
}{}
\DeclareFontShape{OMX}{MnSymbolE}{b}{n}{
	<-6>   MnSymbolE-Bold5
	<6-7>  MnSymbolE-Bold6
	<7-8>  MnSymbolE-Bold7
	<8-9>  MnSymbolE-Bold8
	<9-10> MnSymbolE-Bold9
	<10-12> MnSymbolE-Bold10
	<12->   MnSymbolE-Bold12
}{}
\let\llangle\@undefined
\let\rrangle\@undefined
\DeclareMathDelimiter{\llangle}{\mathopen}%
{MnLargeSymbols}{'164}{MnLargeSymbols}{'164}
\DeclareMathDelimiter{\rrangle}{\mathclose}%
{MnLargeSymbols}{'171}{MnLargeSymbols}{'171}
\makeatother

\hypersetup{
	colorlinks,   
	linkcolor={red!10!black},   
	citecolor={blue!10!black},   
	urlcolor={blue!10!black}
}

\newtheorem{thm}{Theorem}[section]

\newtheorem{definition}[thm]{Definition} 
\newtheorem{lemma}[thm]{Lemma}
\newtheorem{proposition}[thm]{Proposition}

\newtheorem{corollary}[thm]{Corollary}
\newtheorem{remark}[thm]{Remark}

\newtheorem{notation}[thm]{Notation}

\theoremstyle{plain}

\newcommand\preceqdot{\mathrel{\ooalign{$\preceq$\cr
  \hidewidth\raise0.225ex\hbox{$\cdot\pkern0.5mu$}\cr}}}

\newcommand{\Rmnum}[1]{\expandafter\@slowromancap\romannumeral #1@}

\newcommand{\ra}{\rightarrow}
\newcommand{\s}[1]{$#1 $}
\newcommand{\mn}{\mathbb{N}}
\newcommand{\mz}{\mathbb{Z}}
\newcommand{\mr}{\mathbb{R}}

\newcommand{\rom}[1]{\uppercase\expandafter{\romannumeral #1\relax}}

\begin{document}

\title{Bounded cohomology of diffeomorphism groups of higher dimensional spheres}

\author{Zixiang Zhou}
\address{School of Mathematical Sciences, Fudan University, Handan Road 220, Shanghai, 200433, China}
\email{zxzhou22@m.fudan.edu.cn}


\keywords{}

\maketitle

\begin{abstract}
    In this paper we prove the vanishing of the bounded cohomology of \s{\text{Diff}^r_+(S^n)} with real coefficients when \s{n\geq 4} and \s{1\leq r\leq \infty}. This answers the question raised in \cite{FNS24}
    for \s{\geq 4} dimensional spheres . 
\end{abstract}

\section*{Introduction}

Discrete bounded cohomology of groups of homeomorphisms and diffeomorphisms on manifolds have been  studied by different approaches recently. On the one hand, the construction of quasimorphisms on the group \s{\text{Diff}_0(S_g)} \cite[Theorem 1.2]{BHW22}, where \s{S_g} is the closed surface with genus \s{g>0}, shows that the dimension of \s{H^2_b(\text{Diff}_0(S_g))} is infinite; On the other hand, Monod and Nariman developed a general method to
compute bounded cohomology of groups of transformation groups of manifolds in \cite{MN23}.  Given a group action on a semi-simplicial complex, if several bounded acyclicity conditions are satisfied, then we can use the quotient complex of this action to compute the bounded cohomology of this group (see \cref{gp to quotient}).
In this framework, they first computed \s{H^\bullet_b(\text{Diff}^r_+(S^1))} and \s{H^\bullet_b(\text{PL}_+(S^1))}, showing they are both generated by a second degree bounded class. 
Moreover, they showed \s{\text{Diff}^r_+(D^2,\partial D^2)} and \s{\text{Diff}^r_+(M\times \mr^n)} are boundedly acyclic. They also proved the vanishing of \s{H^\bullet_b(\text{Diff}^r_+(S^n))} in low degrees.
Following their study, a more recent paper \cite[Theorem A,B]{FNS24} showed \s{\text{Diff}^r_+(\mr^n)} and \s{\text{Homeo}_+(D^n,\partial D^n)} are boundedly acyclic. So the bounded cohomology of \s{\text{Diff}(\mr^n)} and  \s{Homeo(\mr^n)} have a very different nature from their ordinary cohomology. As an application, they provided the unboundness of the characteristic class of flat \s{\mr^n}-bundles. 


Our main result in this paper is the computation of the bounded cohomology of groups of orientation-preserving diffeomorphisms or homeomorphisms on  high-dimensional spheres, which answers \cite[Question 6.7]{FNS24} for higher dimensions:

\begin{thm}
 \s{\text{Diff}^r_+(S^n)} is boundedly acyclic for \s{n\geq 4,1\leq r\leq \infty}.
\end{thm}

This implies, by \cref{coame to inj}, that \s{\text{Diff}^r(S^n)} is also bounded acyclic. Hence the comparison map \s{H^\bullet_b(\text{BDiff}^\delta_{r}(S^n))\ra H^\bullet(\text{BDiff}^\delta_{r}(S^n))} is zero, where \s{"\delta"} means we view it as a discrete subgroup, this has an immediate application that answers \cite[Question 7.4]{MN23} for higher dimensions:

\begin{corollary}
    Every nontrivial class  of \s{H^\bullet(\text{BDiff}^\delta_{r}(S^n))} in positive degree is unbounded for \s{n\geq 4,1\leq r\leq \infty}.
\end{corollary}

Our strategy, inspired by \cite[Theorem 1.10]{MN23}, is to show that for every \s{k>0}, every subset \s{F_k} with \s{k} elements of \s{S^n}, the group \s{\text{Diff}^r_c(S^n-F_k)} is boundedly acyclic.  Note that this group is coamenable to it's subgroup \s{\text{Diff}^r_c(S^n-U_k)}, where \s{U_k} is a closed tubular neighborhood of the \s{0}-dimensional manifold \s{F_k}, \ie, \s{U_k} is a disjoint union of \s{k} closed balls each containing a point in \s{F_k}. We denote \s{\text{Diff}^r_c(S^n-U_k)} by \s{G_k}.

To prove the bounded acyclicity of \s{G_k},
     we construct a semi-simplicial \s{G_k}-complex \s{Y_\bullet}, which  is the complex associated to a generic relation on its \s{0}-skeleton. We show this complex and the quotient complex as well as  every stabilizer of this action are all boundedly acyclic, thus by a spectral sequence argument (\cref{gp to quotient}), we can show \s{G_k} is boundedly acyclic. 
 Our techniques to prove the bounded acyclicity results are rooted in differential topology.
 We observe that in high dimensions, the diffeomorphisms of spheres that have a given germ at some arcs and disks can be chosen with sufficient freedom.
 Approximations of maps by embeddings , isotopy extensions and transversality theorems are  critical to our findings.

 Our method is not effective in the case when \s{n=2,3}. When \s{n=2}, we cannot define the generic relation in the same way.
 Though we can define another generic relation, for example, let arcs be  transverse as in \cite{MN23}. At this time, the computation of bounded cohomology of the quotient complex seems to be difficult. When \s{n=3}, we cannot get the desired isotopy results, and the stabilizers might not be boundedly acyclic. However, our construction seems to be hopeful to determine the bounded cohomology of \s{Homeo_+(S^n)} answering \cite[Question 7.4]{MN23} in higher dimensions.

\begin{notation} We will keep the following notations in the whole paper:
\begin{itemize}

\item In this paper, we always denote by \s{I} the unit interval \s{ [0,1]}.
     
    \item Given a map \s{f}, we write \s{im(f)} as the image of \s{f}. If \s{f} maps \s{X} to itself, then write \s{supp(f)} as the \textbf{closure} of set of points in \s{X} that is not fixed by \s{f}, namely the \emph{support of \s{f}}. Given a set \s{F} of self maps of \s{X}, then define the \emph{support of \s{F}} as \s{supp(F)=\bigcup_{f\in F}\overline{supp(f)}}.
    \item Given a manifold \s{M}, we denote by  \s{int(M)} the interior of \s{M}, and write the boundary of \s{M} as \s{\partial M}.
    
    \item Given \s{C^r}-manifolds \s{M,N}, let \s{C^r(M,N)} be the set of \s{C^r}-maps from \s{M} to \s{N}. 

    \item For a differentiable map \s{f} between differentiable manifolds, denote by \s{Df} the differential of \s{f} between the tangent bundles. 
    \item Given a map \s{f} and a subset  \s{A\subset dom(f)}, then we write \s{f_A} as the restriction map of \s{f} on \s{A}.
    \item Given a smooth oriented manifold \s{M}. Let \s{\text{Diff}^r_{+}(M)} be the the group of orientation-preserving \s{C^r}-diffeomorphisms of \s{M};
    Let \s{\text{Diff}^r_{c}(M)} be the group of  \s{C^r}-diffeomorphisms with compact support in the \textbf{interior} of \s{M}.

   \item Given \s{f\in C^r(U,V)} where \s{U,V} are open sets of some Euclidean spaces, given \s{0\leq l\leq r, x\in U}, we denote by \s{|D^l(f)(x)|} the maximum of the norms of all \s{l}-order partial derivatives of \s{f} at \s{x}.    
\end{itemize}
    
\end{notation}

\section*{Acknowledgenment}
The author is indebted to his advisor Xiaolei Wu, who patiently proceeded the proofreading and provided many useful suggestions for the preparation of the paper. He is also grateful to his friend Josiah Oh, who pointed out some errors in the paper; And his teacher Guozhen Wang
, who pointed out many useful sources and facts about topological manifolds him the realted questions. He also thanks 
Francesco Fournier-Facio, Nicolas Monod, Sam Nariman for useful discussions.

\section{Tools in differential topology and bounded cohomology}

In this section we will assume \s{1\leq r\leq \infty} unless otherwise specified.
\subsection{Tools in bounded cohomology: coamenable groups and group actions.}

Let $X_{\bullet}$ be a semi-simplicial set. Its bounded cohomology $H_{b}^{n}\left(X_{\bullet}\right)$ is the cohomology of the complex
$$
0 \rightarrow \ell^{\infty}\left(X_{0}\right) \rightarrow \ell^{\infty}\left(X_{1}\right) \rightarrow \ell^{\infty}\left(X_{2}\right) \rightarrow \cdots
$$
where the coboundary operators are defined as the duals of the face maps of $X_{\bullet}$. The coefficients are always understood to be $\mathbb{R}$ unless specified otherwise. We say that $X_{\bullet}$ is \emph{boundedly acyclic} if $H_{b}^{n}\left(X_{\bullet}\right)=0$ for all $n>0$.

If $G$ is a discrete group and $X_{\bullet}$ is the nerve of $G$, this defines the bounded cohomology of $G$, denoted by $H_{b}^{n}(G)$. More explicitly, we have $X_{\bullet}=G^{\bullet}$, with face maps defined as
$$
\begin{aligned}
& d_{i}: G^n \rightarrow G^{n-1} \\
& d_{0}\left(g_{1}, \ldots, g_{n}\right)=\left(g_{2}, \ldots, g_{n}\right) \\
& d_{i}\left(g_{1}, \ldots, g_{n}\right)=\left(g_{1}, \ldots, g_{i-1}, g_{i} g_{i+1}, g_{i+2}, \ldots, g_{n}\right), \quad i \in\{1, \ldots, n-1\} \\
& d_{n}\left(g_{1}, \ldots, g_{n}\right)=\left(g_{1}, \ldots, g_{n-1}\right)
\end{aligned}
$$

We say that the group $G$ is boundedly acyclic if $H_{b}^{n}(G)=0$ for all $n>0$.

Now we first prove a theorem which is a generalization for \cite[Theorem 1.6]{MN23}:

\begin{lemma}\label{ba for product}
    For any smooth manifold \s{M}, \s{1\leq r\leq \infty}, the group \s{\text{Diff}^r_{c}(M\times \mr)} is boundedly acyclic.
\end{lemma}
\begin{proof}
    The idea is to construct a \s{\mz-}commuting conjugate \cite[Theorem 1.3]{CFFLM23} for every finitely generated subgroup \s{\Gamma}. First, since \s{\Gamma} is finitely generated, we have \s{supp(\Gamma)\subset K\times J} where \s{K} and \s{J} are compact. Pick a neighborhood \s{K'}(resp. \s{J'}) of \s{K}(resp.\s{J}) with compact closure and a smooth map \s{\psi: M\ra [0,1]}(resp.\s{\varphi: \mr\ra [0,1]}) such that
    \s{\psi_K\equiv 1} and \s{\psi_{M-K'}\equiv 0}(resp.\s{\varphi_J\equiv 1} and \s{\varphi_{M-J'}\equiv 0}).
    . We can construct a smooth vector field \s{X} on \s{M\times \mr} by \s{X_{(k,t)}=(0,\psi(k)\varphi(t))} for all \s{(k,t)\in M\times \mr}. This is a bounded supported hence complete vector field. Let \s{(\Psi_t)_{t\in\mr}} be the associated flow, then for all \s{t>|J|}, \s{\Psi_t(K\times J)} is disjoint from \s{K\times J}. Hence 
    \s{\Psi_t^m(K\times J)=\Psi_{tm}(K\times J)} is disjoint from \s{K\times J}, which means
    \s{\Psi_t} is a \s{\mz-}commuting conjugate for \s{\Gamma}. Hence the result follows from \cite[Theorem 1.3]{CFFLM23}.
\end{proof}

A group $G$ is \emph{coamenable} to its subgroup  $H$ (or say \s{H} is coamenable in \s{G}) if there is a $G$-invariant mean on $G / H$, apparently every group is coamenable to it's finite index subgroup. the following propositions will be used:

\begin{proposition}\cite[Proposition 10]{Mon22}\label{conju to coame}
  Let $G$ be any group and $G_{0}<G$ a subgroup.
If every finite subset of $G$ is contained in some conjugate of $G_{0}$, then $G_{0}$ is co-amenable in $G$.

\end{proposition}

\begin{proposition}\cite[Proposition 11]{Mon22}\label{coame to inj}
    Given any group \s{G}, if \s{G} is coamenable to its subgroup \s{G_0} then the natural map induced by the inclusion:
$$
{H}_{{b}}^{\bullet}(G) \rightarrow {H}_{{b}}^{\bullet} (G_{0})
$$
is injective. Therefore, if \s{G_0} is boundedly acyclic, then so is \s{G}.
\end{proposition}

\begin{notation}
    \begin{itemize}
    \item  A \s{m}-dimensional \s{C^r}-manifold (resp. \s{C^r}-manifold with \s{k}-corners) is defined by the \s{C^r}-charts from open subsets of \s{\mr^{m}\times\mr_{\geq 0}} (resp. \s{\mr^{m-k}\times\mr_{\geq 0}^{k}}).  
        \item Given \s{C^r}-manifolds \s{M} and \s{N}, we say a map \s{f: M\ra N} is a \emph{\s{C^r}-embedding} if \s{f} is a \s{C^r}-diffeomorphism onto its image. In this case we say \s{f(M)} is a submanifold of \s{N}. 
If \s{\partial M,\partial N\neq \varnothing} and \s{f} moreover maps the boundary to the boundary, the interior to the interior, and is transversal to \s{\partial N} then we say \s{f} is a \emph{neat} embedding. When \s{M} is a submanifold of \s{N} where the inclusion map is a neat embedding, we say \s{M} is a \s{C^r}-embedded \emph{neat} submanifold of \s{N}. It follows that when \s{M\subset N} is an \s{C^r}-embedded submanifold both with empty boundaries, or \s{M\subset N} is a \s{C^r}-embedded neat submanifold, then \s{M} has a tubular neighborhood in \s{N}.

\item Given \s{C^r}-manifolds \s{M, N} and a \s{C^r}-map \s{F:I\times M\ra N}, we say it's an \s{C^r}-isotopy if \s{F_s:=F(s,-): M\ra N} is an embedding for each \s{s\in I}.
\item Given a \s{C^r}-isotopy \s{F:I\times M\ra N}, if \s{N=M}, then we define the \emph{support of \s{F}} \s{supp(F)} to be the union of the support of all \s{F_s}.
    \end{itemize}
\end{notation}

The following is our main tool to prove the bounded cohomology of a group: 

\begin{lemma}\cite[Theorem 3.3]{MN23}\label{gp to quotient}
Let \s{G} be any group,
    if \s{G} acts on a boundedly acyclic connected semi-simplicial complex \s{X_\bullet} such that the set of \s{k}-stabilizers of this action is uniformly boundedly acyclic for each \s{k\geq 1}. Then \s{H^\bullet_b(G)\cong H^\bullet_b(X_\bullet/G)}.
\end{lemma}

The following is a general technique to construct a coamenable subgroup in the diffeomorphism groups of a manifold.

\begin{lemma}\label{deformcoame}
    Let \s{M} be a \s{C^r}-manifold, \s{A} be a \s{C^r}-embedded  closed submanifold (or neatly embedded submanifold) of \s{M}, then there exists a closed tubular neighborhood \s{V} of \s{A} such that \s{\text{Diff}^r_{c}(M-A)} is coamenable to \s{\text{Diff}^r_{c}(M-V)}.
\end{lemma}
\begin{proof}
Pick a tubular neighborhood \s{U} of \s{A}. Identifying \s{U} with the normal bundle \s{\nu A} of \s{A} equipped with a Riemann metric, \s{U} consists of all the pairs \s{(x,p)} where \s{x\in A,p\in \nu_x}. For any \s{\epsilon>0}, let \s{U_\epsilon} be the subset of all the pairs \s{(x,p)\in U_1} such that \s{|p|<\epsilon}.
Let \s{V:=U_1}

We only need to prove that for each finite set \s{F} of \s{\text{Diff}^r_{c}(M-A)}, there exists \s{\gamma\in \text{Diff}^r_{c}(M-A)} that \s{\gamma(supp(F))\subset M-V}, since then \s{\gamma F\gamma^{-1}\subset \text{Diff}^r_{c}(M-V)} and the results follows from \cref{conju to coame}.

Since the interior of \s{supp(F)} is disjoint from \s{A}, there exists \s{\delta>0} such that   \s{U_\delta}  is disjoint from \s{supp(F)}. Now pick a smooth diffeomorphism \s{\psi:\mr_{\geq0}\ra \mr_{\geq 0}} such that \s{\psi} has the trivial germ at \s{0} and \s{\infty}, and \s{\psi(\delta)>1}, define a diffeomorphism \s{\gamma} on \s{U} as \s{\gamma(x,p)=(x,\psi(|p|)p)}. It's easy to see that \s{\gamma} is compactly supported and orientation-preserving, moreover \s{supp(\gamma)\subset U-A}. Hence extending \s{\gamma} trivially elsewhere we get  \s{\gamma\in \text{Diff}^r_{c}(M-A) } and \s{\gamma(supp(F))\subset M-V}.
\end{proof}

\subsection{Tools in differential topology: approximations, embeddings, isotopies and transversalities.}


Let \s{M,N} be two smooth manifolds possibly with boundaries, where \s{M} is  compact.
Given two family of charts \s{\Phi=(\varphi_i,U_i)_{i\in \Lambda}} of \s{M} and \s{\Psi=(\psi_i,V_i)_{i\in \Lambda}}  of \s{N} and a family \s{K=(K_i)_{i\in\Lambda}} of compact sets \s{K_i\subset U_i} where \s{\Lambda} is finite and \s{\bigcup_{i\in\Lambda} int(K_i)=M}
such that \s{f(K_i)\subset V_i},
we can define a neighborhood 
$$
\mathscr{N}^{r}=\mathscr{N}^{r}(f ; \Phi, \Psi, K, \varepsilon)
$$
to be the set of $C^{r}$ maps $g: M \rightarrow N$ such that for all $i \in \Lambda, g (K_{i} ) \subset V_{i}$ and
$$
 |D^{l} (\psi_{i} f \varphi_{i}^{-1}-\psi_{i} g \varphi_{i}^{-1} )(x) |<\varepsilon
$$
for all $x \in \varphi_{i} (K_{i} ), l=0, \ldots, r$. 
The set of of such neighborhood form a basis of a topology on \s{C^r(M,N)}, called \emph{strong topology} of \s{C^r(M,N)} \cite[Chapter 2.1]{Hir76}. From now on when we talk about topology on \s{C^r(M,N)} we are always referring to this topology. We call the neighborhood \s{\mathscr{N}^r} of this form a \emph{strong neighborhood}.
We can also define a metric on this strong neighborhood \s{\mathscr{N}^r} as $$d_{\mathscr{N}^r}(g,h)=sup_{0\leq l\leq r,i\in \Lambda, x\in\varphi_i(K_i)} |D^{l} (\psi_{i} f \varphi_{i}^{-1} -\psi_{i} g \varphi_{i}^{-1} )(x) |$$
Note that on \s{\mathscr{N}^r} the topology defined by this metric coincides with the strong subspace topology on \s{\mathscr{N}^r}. We simply say a variable \s{g} in \s{C^r(M,N)} approximates \s{f\in C^r(M,N)}, if it approximates to \s{f} in the sense of the strong topology.

The main purpose of the rest of this section is to introduce some classical differential topology results, and eventually prove the main lemma \cref{annuisoto}.

\begin{proposition}\cite[Theorem 2.1.4]{Hir76}\label{emb open and dense}
Given compact smooth manifolds \s{M,N},  let \s{Emb^r(M,N)} be the set of \s{C^r}-embeddings from \s{M} to \s{N}. Then
    \s{Emb^r(M,N)} is open in \s{C^r(M,N)}. 
\end{proposition}

Suppose that \s{\partial M, \partial N} are non-empty, let \s{C^r(M,N;  \partial)} be the set of \s{C^r}-maps from \s{M} to \s{N} that maps \s{\partial M} to \s{\partial N}, let \s{Emb^r(M,N; \partial)} be the set of neat \s{C^r}-embeddings from \s{M} to \s{N}. They are subspaces in \s{C^r(M,N)}. Then \s{Emb^r(M,N; \partial)} is open in \s{C^r(M,N; \partial)}. In fact,  \s{Emb^r(M,N; \partial)} is the intersection three sets: the set of \s{C^r}-maps that are transversal to \s{\partial N},  \s{Emb^r(M,N)} and \s{C^r(M,N; \partial)}. The first two are open maps.

\begin{proposition}\cite[Theorem 2.2.13]{Hir76}\label{emb dense}
Given smooth manifolds \s{M,N} with \s{M} closed. If \s{dim(N)\geq 2dim(M)+1}, then \s{Emb^r(M,N)} is dense in \s{C^r(M,N)}.
\end{proposition}

The following shows that if two maps are close enough in \s{C^r(M,N)} then they are homotopic and the homotopy can be chosen within a controlled distance from these maps:

\begin{proposition}\label{isotopy in close map}
Given compact smooth manifolds \s{M,N}.
    Let \s{f\in C^r(M,N)} and \s{\mathscr{N}^r} be a strong neighborhood of \s{f} in \s{C^r(M,N)}. Then : \begin{itemize}
    \item There exists a strong neighborhood \s{\mathscr{N'}^r\subset \mathscr{N}^r}, such that for any \s{f'\in \mathscr{N'}^r}, there is a \s{C^r}-homotopy \s{F:I\times M\ra N} such that \s{F_0=f,F_1=f'} and \s{F_s\in \mathscr{N}^r} for every \s{s\in  I};
        \item If \s{f} is an embedding, then \s{F_s} can be chosen to  be a \s{C^r}-isotopy; 
        \item If \s{f} and \s{f'} coincide at some subset \s{Z\subset M}, then \s{F_s} can also be chosen to coincide with \s{f} at \s{Z}.
    \end{itemize}

\end{proposition}

\begin{proof}
    By the Whitney embedding theorem, there is \s{m\in\mn} such that we can smoothly embed \s{N} into \s{\mr^m} if \s{N} has no boundary or smoothly neatly embed \s{N} into \s{\mr^{m-1}\times \mr_{\geq 0}} if \s{N} has boundary. we only prove the first case where \s{N} has no boundary, the other cases can be done similarly. Now \s{N} has a compact tubular neighborhood \s{U} and let \s{\pi_N:U\ra N} be the projection map. 
    
  Observe \s{C^r(M,N)} is a subspace of \s{C^r(M,\mr^m)}. Moreover by  compactness of \s{M}, the strong topology of \s{C^r(M,N)} actually coincides with the subspace topology in \s{C^r(M,\mr^m)}.
    
    Hence there is a  strong neighborhood \s{\mathscr{K}^{r}=\mathscr{K}^{r}(f ; \Phi, id_{\mr^m}, K, \varepsilon)} of \s{f} in  \s{C^r(M,\mr^m)}, such that \s{\mathscr{K}^r\cap C^r(M,N)\subset \mathscr{N}^r} and for each \s{f'\in \mathscr{K}^r} the line segment \s{[f'(x),f(x)]} is contained in \s{V} for every \s{x\in M}.

    Now we can define the \s{C^r}-homotopy \s{F_s:M\ra N,s\in  I} as follows:$$
    F_s(x)=\pi_N\bigg(sf(x)+(1-s)f'(x)\bigg)
    $$
    By definition we have that \s{d_{\mathscr{K}^r}(sf+(1-s)f'),f)\leq d_{\mathscr{K}^r}(f,f')}. Now if we choose \s{f'} close enough to \s{f},  then \s{d_{\mathscr{K}^r}(F_s,f)=d_{\mathscr{K}^r}(F_s,\pi_Nf)\leq C(r)C(\pi_N)d_{\mathscr{K}^r}(f,f')}. Here \s{C(r)} is a constant that only depends on \s{r}, and \s{C(\pi_N)} is the maximum value of all \s{|D^l (\pi_N)|} restricting to a \s{\epsilon}-neighborhood of \s{f(M)}, hence is finite. so there is strong neighborhood \s{\mathscr{K'}^r} such that whenever \s{f'\in \mathscr{K'}^r}, the defined \s{F_s} will lie in \s{\mathscr{K}^r} for all \s{s\in I}. Now let \s{\mathscr{N}^r=\mathscr{K}\cap C^r(M,N)} and \s{\mathscr{N'}^r=\mathscr{K'}^r\cap C^r(M,N)},
 the first  and the third  statement follow.

    Now the second statement follows from \cref{emb open and dense}.
\end{proof}

Similarly we also have an analogy of \ref{isotopy in close map} in the relative case: If \s{f\in C^r(M,N;\partial )} and  \s{\mathscr{N}^r} is a strong neighborhood of \s{f} in \s{C^r(M,N;\partial)}, then we can choose a subset \s{\mathscr{N}'^r\subset \mathscr{N}^r} open in \s{\subset C^r(M,N;\partial)} so that: for every \s{f'\in \mathscr{N}'^r} there is a \s{C^r}-homotopy  \s{F_s\in C^r(M,N;\partial)} with \s{F_0=f,F_1=f'} and \s{F_s\in \mathscr{N}^r} for every \s{s\in I}; If \s{f} is moreover a neat embedding, then the chosen \s{F_s} is also a neat embedding for every \s{s\in I}. To get the last statement one only need to notice that \s{Emb^r(M,N; \partial)} is open in \s{C^r(M,N; \partial)}.

\begin{proposition}\label{dense of smooth map}
    
Let $M$ and $N$ be compact smooth manifolds. Then $C^\infty(M, N)$ is dense in $C^{r}(M, N)$. Moreover, let \s{A} be a closed subset of \s{M}, if \s{f\in C^{r}(M, N)} is already smooth on a neighborhood \s{U} of  \s{A}, then \s{f} can be approximated by a smooth map \s{f'} that coincides with \s{f} on a neighborhood of \s{A}. 
\end{proposition}
\begin{proof}
The first statement is from \cite[Theorem 2.3.3]{Hir76}. Now we prove the second.
First fix a smooth  map \s{\xi: M\ra I}, such that \s{\xi} is \s{1} on a neighborhood of \s{A} and \s{0} on \s{M-U}. 
  Now look at the function \s{h(x):=\pi_N\bigg(\xi(x)f'(x)+(1-\xi(x)f(x))\bigg)} where \s{\pi_N}  is the same as before.
    Now let \s{f'} vary in \s{C^\infty(M,N)}. Since \s{\pi_N} is smooth and \s{f(M)} is compact, if \s{f'} approximates \s{f}, then \s{h} will also approximate \s{f}. Now \s{h} is smooth by definition, and stays close to \s{f}.
\end{proof}

 Let \s{C^r(M, \text{near } \partial M)} be the subspace of \s{C^r(M,M)} consisting of the elements that is identity near \s{\partial M}, where we give \s{C^r(M, \text{near } \partial M)} the subspace topology of \s{C^r(M,M)}. The following shows the openness of diffeomorphism groups of manifolds with boundary, which is a generalization of the classical theorem for manifolds without boundary \cite[Theorem 2.2.7]{Hir76}.

\begin{lemma}\label{diff op}
  Let  \s{M} be a connected smooth compact manifold.  Then \s{ \text{Diff}^r_c(int(M))} is open in \s{C^r(M,\text{near } \partial M)}, the space of \s{C^r}-self maps of \s{M} with trivial germ at \s{\partial M}.
 
\end{lemma}
\begin{proof}
    Let \s{f\in \text{Diff}^r(M, \text{near } \partial M)} and let \s{f'} vary in \s{C^r(M,\text{near } \partial M)}. When \s{f'} approximates \s{f}, \s{f'} becomes an embedding. Moreover, since \s{f'} is identity near \s{\partial M}, \s{f'} is a local diffeomorphism. Hence \s{f'(M)} is both open and closed in \s{M}, which equals \s{M}. So \s{f'} is a diffeomorphism, as required.
\end{proof}

Hence for \s{1\leq r\leq \infty} every \s{C^r}-embedding can be approximated by a smooth embedding, every \s{C^r}-diffeomorphism can be approximated by a smooth diffeomorphism.

\begin{lemma}\label{openess of isotopy}
Given compact smooth manifolds \s{M,N},
    the set of \s{C^r}-isotopies \s{F:I\times M\ra N} is open in \s{C^r(I\times M,N)}.
\end{lemma}
\begin{proof}
    The isotopy \s{F} gives a continuous path \s{c_F:I\ra Emb^r(M,N)}. We can find a big enough positive number \s{m} and a family of strong neighborhood \s{\{\mathscr{N}^r_i\}_{1\leq i\leq m}} in \s{C^r(M,N)} so that for each \s{1\leq i\leq m}, \s{\mathscr{N}^r_i} contains \s{c_F([\frac{i-1}{m},\frac{i}{m}])} and \s{\mathscr{N}^r_i\cap C^r(M,N)\subset Emb^r(M,N)}, define a subset \s{\mathscr{N}^r\subset C^r(I\times M,N)} as the set of \s{G:I\times M\ra N} such that \s{c_G([\frac{i-1}{m},\frac{i}{m}])\subset \mathscr{N}^r_i} for each \s{1\leq i\leq m}. This is an open subset and each element of it is an isotopy.
\end{proof}

The following theorem allows us to push a map between manifolds so that its image is disjoint from some given disks in the manifold.

\begin{lemma}\label{away from disk}Let \s{M,N} be compact smooth manifolds  with \s{dim N=n,{dim(M)<dim(N)}}.
    Let \s{f:M\ra N} be a \s{C^r}-map. Let \s{K} be the disjoint union of finitely many \s{n}-dimensional \s{C^r}-embedded compact disks lying in  \s{int(N)}. Let \s{U} be any neighborhood of \s{K}. Then there exists \s{g\in \text{Diff}^r_0(N)} with \s{supp(g)\subset int(U)} such that \s{gf(M)} is disjoint from \s{K}. 
\end{lemma}
\begin{proof}
    For simplicity we only prove the case when \s{K} is a single disk, and the general cases can be proved similarly.

    By definition,\s{f(M)\cap K} is nowhere dense and closed in \s{K}, so there exist another closed disk \s{K'}  lying in the interior of \s{K} such that \s{K'\cap f(M)=\varnothing}. By 'blowing up' \s{K'} to \s{K}, we can easily find \s{g\in \text{Diff}^r_0(int(U))} with compact support in \s{int(U)} such that \s{g(K')=K}. Extending \s{g} trivially elsewhere in \s{M}, we get \s{g\in \text{Diff}^r_0(M)} hence \s{gf(M)\cap K=g\bigg(f(M\cap K')\bigg)=\varnothing}.
\end{proof}

The following transversality theorem allows us to perturb a map between manifolds so that its image is disjoint from some low dimensional embedded submanifolds. Recall that a map \s{f\in C^r(M,N)} is said to transverse to a submanifold \s{A\subset N}, if whenever \s{f(x)=y\in A}, the tangent space of \s{N} at \s{y} is spanned by the tangent space of \s{A} at \s{y} and the image of the tangent space of \s{M} at \s{x}. When we say two maps\emph{ agree on a neighborhood of a set}, we mean their germs on this set are the same.

\begin{lemma}\label{emb and trans} Let \s{M,N} be compact smooth manifolds.
    Let \s{f:M\ra N} be a \s{C^r}-map. Let \s{M'\subset M} be a closed set and \s{A} be a submanifold (or neat submanifold) of \s{N} such that \s{f(M')\cap A=\varnothing}. Then for any strong neighborhood \s{\mathscr{N}^{r}=\mathscr{N}^{r}(f ; \Phi, \Psi, K, \varepsilon)} of \s{f}, there exists   \s{f'\in\mathscr{N}^r} such that \s{f'} and \s{f} have the same germ at \s{M'} and \s{f'} is transverse to \s{A}. 
\end{lemma}

\begin{proof}
As before we assume \s{N} is embedded into a Euclidean space (or a half Euclidean space).
Pick a neighborhood \s{U} of \s{M'} such that \s{\overline{f(U)}} is disjoint from \s{A}. 
There is no harm to assume \s{\mathscr{N}^r} to be smaller, so we assume \s{\overline{g(U)}} is disjoint from \s{A} for every \s{g\in \mathscr{N}^r}.
Pick a smooth map \s{\psi:M\ra  I} such that \s{\psi_{M'}\equiv1} and \s{\psi_{M-U}\equiv 0}. 
Pick a strong neighborhood \s{\mathscr{N}'^r\subset C^r(M,N)} for \s{\mathscr{N}^r} as in \cref{isotopy in close map}.

By the transversality theorem \cite[Theorem 3.2.1]{Hir76}, the set of maps  that are transverse to \s{A} on \s{M} is dense in \s{C^r(M,N)}, so there exists \s{f_1\in\mathscr{N'}^r} transversal to \s{A}. 
    Now define \s{f'(x):=\pi_N\left(\psi(x)f(x)+(1-\psi(x))f_1(x) \right)} where \s{\pi_N} is the same as in \cref{isotopy in close map}. Since \s{\pi_N} is smooth and \s{f(M)} is compact, if \s{f_1} approximates \s{f} in \s{C^r(M,N)}, then the map \s{x\mapsto\psi(x)f(x)+(1-\psi(x))f_1(x)} also approximates \s{f} in \s{C^r(M,\mr^m)}, thus \s{f'} approximates \s{f}. So by choosing \s{f_1} close enough to \s{f} in \s{C^r(M,N)}, we may assume \s{f'\in \mathscr{N}^r}.
Now by definition, we have \s{\overline{f(U)}} is disjoint from \s{A} and \s{f'_{M-U}} is transverse to \s{A}, hence \s{f'} is transverse to \s{A}.
\end{proof}

Let \s{M} be a compact smooth manifold. Given a isotopy \s{F: I\times V\ra M },  we can define an embedding 
$\hat{F}: I\times V\ra I\times M$ as
$\hat{F}(t,x)  \mapsto\bigg(t,F(t,x)\bigg)$. The following lemma is a variant of the classical isotopy extention lemma that allow us to construct an ambient isotopy with controlled support:

\begin{lemma}\label{iso-ex}
    Let \s{M} be a smooth manifold and 
let $V \subset M$ be a compact neat submanifold. And let  $F: I\times V \rightarrow M$ an \s{C^r}-isotopy of $V$ such that \s{F(I\times int(V))\subset int(M)} and $F(I\times \partial V) \subset \partial M$ and \s{F_{t\times V}} is transversal to \s{\partial M} for each \s{t\in M}, \ie, \s{F} is a \s{C^r}-map such that \s{F_s} is a neat \s{C^r}-embedding for each \s{s\in I}. Let \s{U} be a neighborhood of \s{supp(F)}, then $F$ extends to a \s{C^r}-diffeotopy of $M$ having compact support in \s{U}.
\end{lemma}
\begin{proof}
    The proof is almost the same with the proof in \cite[Theorem 8.1.4]{Hir76}, we just cite it and make a slight modification here.
    
We start from the vector field $X$ on $\hat{F}(I\times V)$ tangent to the arcs $\hat{F}(I\times  x)$. The condition of \s{F} in the lemma ensures that \s{\hat{F}(I\times  x)} is a neat embedding into a manifold with corners. Thus it gives us a tubular neighborhood of \s{\hat{F}(I\times  x)} in \s{I\times M}. 
By means of a partition of unity, the horizontal part of $X$ is extended to a vector field $Y$ with support in \s{I\times U} on a neighborhood of $\hat{F}(I\times V)$ in $ I\times M$. The hypothesis on $F$ allows us to assume that $Y_{(t,x)}$ is tangent to $ I\times (\partial M)$ whenever $x \in \partial M$. After restricting to a smaller neighborhood, the horizontal part of $Y$ is extended to a time-dependent vector field $G$ on $M$ with support in \s{U}. The diffeotopy generated by $G$ completes the proof.
\end{proof}

Let \s{n\geq 4}.
  In the following we view \s{0} as the origin of \s{\mr^n} and \s{S^{n-1}} is viewed as \s{\mr^{n-1}\cup\{\infty\}}. 
  Given a oriented smooth manifold \s{M}, let \s{Fr^+(M)} be the oriented \s{n}-frame bundle of the tangent bundle \s{TM}. Then there is a fiber bundle $
  GL^+_n(\mr) \hookrightarrow  Fr^+(M) \twoheadrightarrow M$. Hence we have an exact sequence of groups: $
  \pi_2(M)\ra \pi_1(GL^+_n(\mr))\ra   \pi_1(Fr^+(M)) \ra \pi_1(M)\ra 1$. If \s{M} is simply connected,  then \s{\pi_1(Fr^+(M))\leq \pi_1(GL^+_n(\mr))=\mz/2\mz}. Hence \s{Aut(\pi_1(Fr^+(M)))} is trivial and the induced \s{\text{Diff}^r_+(M)}-action on \s{\pi_1(Fr^+(M))} is trivial. Moreover, this means given any path \s{c: I\ra Fr^+(M)} and any \s{f\in \text{Diff}^r_+(M)} such that \s{f_*c(x)=c(x)} when \s{x\in \{0,1\}}, then \s{f_*c} is homotopic to \s{c} rel \s{\{0,1\}}. Now if moreover both \s{im(c)} and \s{im(fc)} are contained in a contractible subset \s{U} of \s{M}, look at maps \s{pc: I\xrightarrow{c} Fr(U)\cong U\times GL^+_n(\mr)\xrightarrow{p} GL^+_n(\mr)} and \s{pf_*c}, we find \s{pc} is homotopic to \s{pf_*c} rel \s{\{0,1\}}.

Now we are ready to prove our main lemma
which will be used to prove \cref{transitivity of orbit}. We define an arc in a manifold as the image of a smooth embedding from \s{I} to this manifold that is transversal to the boundary, it may or may not be a neat embedding.

\begin{lemma}[Main lemma]\label{annuisoto}
    Let \s{f\in \text{Diff}^r_{c}(S^{n-1}\times I)} and \s{K\subset S^{n-1}\times I} be such that:

    \begin{itemize}
       
        \item \s{K} is a union of finitely many pair-wise disjoint \s{n}-dimensional smoothly embedded compact disks lying in the interior and finitely many  arcs in \s{S^{n-1}\times I};
        \item \s{K} is disjoint from \s{0\times I\cup f(0\times I)}.
    \end{itemize} 
    Then there is  \s{\varphi\in \text{Diff}^r_{c}(S^{n-1}\times I)} with the same germ as \s{f} on \s{0\times I} and with the trivial germ at \s{K}.
   
\end{lemma}
\begin{proof}
\textbf{Step 1: Reduce to the special case where \s{f} has a smooth germ at \s{0\times I}.}
Choose a compact neighborhood \s{U} of \s{0\times I} such that \s{f(U)\cap K=\varnothing} and a smooth function \s{\xi:S^{n-1}\times I\ra [0,1]} that is \s{1} on a neighborhood of \s{0\times I} and \s{0} outside \s{U}. For any \s{g\in C^\infty(S^{n-1}\times I, S^{n-1}\times I)} we define
$$h(x)=\pi_{S^{n-1}\times I}\bigg(\xi(x)f(x)+(1-\xi(x))g(x)\bigg)$$ 
where \s{\pi_{S^{n-1}\times I}} is the projection map from some tubular neighborhood of \s{S^{n-1}\times I} in some Euclidean space, the same notation as in \cref{isotopy in close map}.
Now let \s{g} vary. By \cref{dense of smooth map}, we can let \s{g} approximate \s{f} while having the trivial germ at \s{S^{n-1}\times \{0,1\}}. Meanwhile, \s{h} also approximates \s{f} and becomes a \s{C^r}-diffeomorphism by \cref{diff op} . Moreover, \s{h} and \s{f} have the same germ at \s{K},  \s{h} has a smooth germ at \s{0\times I} and \s{h(U)} is disjoint from \s{K}.
Therefore  \s{hf^{-1}} has the trivial germ at \s{K}.
So be composing \s{hf^{-1}} on \s{f}, we may assume \s{f} has the smooth germ at \s{0\times I}.

\textbf{Step 2: Reduce to the special case where \s{f} fixes \s{0\times I}.} In the rest of the proof we assume all of the maps mentioned are smooth unless otherwise specified.
    Let \s{\pi:S^{n-1}\times I\ra S^{n-1}}, \s{\pi':S^{n-1}\times I\ra I} be the projections, and \s{f_I} be  the restriction of \s{f} on \s{0\times I}. Since by definition \s{f} has the trivial germ at \s{S^{n-1}\times \{0,1\}}, we may assume $$f_I(t)=(0,t)\in S^{n-1}\times I, t\in [0,\epsilon]\cup [1-\epsilon,1] $$ for some \s{\epsilon>0}. Now  \s{\pi f_{[\epsilon,1-\epsilon]}} can be approximated by an   map \s{c:I\ra S^{n-1}} such that: \begin{itemize}
        \item \s{c_{[\epsilon,1-\epsilon]}} is an embedding;
        \item \s{c(t)=0} when \s{t\in [0,\epsilon]\cup [1-\epsilon,1] };
        \item  All derivatives of \s{c} at \s{\{\epsilon,1-\epsilon\}} are \s{0};
        \item \s{\infty\notin im(c)}.
    \end{itemize} 

    The map \s{c} can be constructed as follows:
    Since \s{\pi f_{0\times [\epsilon,1-\epsilon]}} is an loop in \s{S^{n-1}} and \s{n-1\geq 3},
    by \cref{emb dense} it's approximated  by an  embedding of \s{S^1}. We use a local translation of \s{\mr^{n-1}} acting on \s{c} to make sure the starting point of the circle is mapped to \s{0}. Next perturb \s{c} so that it is transversal to hence disjoint from \s{\infty}.
    Then we reparameterize this embedded circle so that its derivatives at the starting point is \s{0}. Finally this embedded loop gives us \s{c_{[\epsilon,1-\epsilon]}} and we are done.
    
    This means \s{f_I} is approximated by an embedding  \s{f':t\mapsto \bigg(c(t),\pi'f(t)\bigg)}. By \cref{isotopy in close map} there is an isotopy \s{F_s:I\ra S^{n-1}\times I} such that \s{F_0=f_I, F_1=f'}. Moreover, by the choice of \s{F_s} we can moreover assume that \s{F_s(t)=(0,t)} when \s{t\in [0,\epsilon]\cup [1-\epsilon,1] }. 

 Now since \s{\infty\notin im(c)}, we can define an map \s{G:[0,3]\times I \ra \mr^{n-1}\times I} as follows: $$
 G_s(t)=\begin{cases}
     F_{\phi(s)}(t)  &s\in  [0,1]\\
     \bigg( c(t), \phi(s-1)t+\Big(1-\phi(s-1)\Big)\pi'f_I(t) \bigg) &s \in [1,2]\\
     (\phi(3-s)c(t),t)  &s\in[2,3]
 \end{cases}
 $$  
 where \s{\phi:I\ra I} is a   map such that \s{\phi_{[0,\frac{1}{3}]}\equiv0} and \s{\phi_{[\frac{2}{3},1]}\equiv1}. We have \s{G_0=f_I} and \s{G_3(t)=(0,t)} for all \s{t\in I}. Since \s{G_s} is injective  for every \s{s\in[0,3]}, \s{G} is a diffeotopy.
 
Now we want to make the image of \s{G} disjoint from \s{K}.
 By \cref{away from disk} we can compose \s{G} with  a diffeomorphism \s{g} with support disjoint from \s{im(G_0)\cup im(G_3)}, such that \s{im(gG)}  is disjoint from  all disks in \s{K}, hence without the loss of generality we may assume \s{im(G)} is disjoint from the disks in \s{K}. Since \s{G(\partial [0,3]\times I)} is disjoint from \s{K}, 
 by \cref{emb and trans} and \cref{openess of isotopy} we could moreover perturb the map \s{G} in the interior of \s{[0,3]\times I} so that \s{G} is transverse to  the arcs in \s{K}, and still disjoint from disks in \s{K}. Since \s{[0,3]\times I} is 2-dimensional and arcs in \s{K} are 1-dimensional and \s{n\geq 4}, we get \s{im(G)} is disjoint from arcs in \s{K}. In conclusion, we may assume that \s{im(G)} is disjoint from \s{K}.
 
 Now  for every \s{s\in[0,3]}, \s{G_s} and the standard embedding \s{t\mapsto (0,t)} have the same germ at the endpoints. Therefore by \cref{iso-ex} we can extend \s{G} to become  an ambient isotopy \s{H:[0,3]\times (S^{n-1}\times I)\ra S^{n-1}\times I}  such that for every \s{s\in [0,3]}, \s{G_s=H_sf_I} and \s{supp(H)} is disjoint from \s{K\cup S^{n-1}\times \{0,1\}}. Hence we have \s{H_3f_I(t)=(0,t)}, now by composing with \s{H_3^{-1}} we may assume \s{f} fixes \s{0\times I}.

\textbf{Step 3: Reduce to the special case where \s{f} has the trivial germ at \s{0\times I}.}
Let \s{\mathscr{P}} be the set of \s{n\times n} matrices with real coefficients of the form $$
\begin{pmatrix}
    A&B\\
    0&1\\
\end{pmatrix}
$$
where \s{A\in GL_{n-1}^+(\mr)} and \s{B\in M_{{(n-1)}\times 1}(\mr)}.

The canonical inclusion \s{GL^+_{n-1}(\mr)\ra \mathscr{P}\ra GL^+_{n}(\mr)} induces the isomorphisms between the fundamental groups.
 Since the induced \s{\text{Diff}^r_+(S^{n-1}\times I)}-action on \s{\pi_1(S^{n-1}\times I)} is trivial, the path defined by \s{p(f):I\ra \mathscr{P},t\mapsto Df_{(0,t)}} is homotopic rel \s{\{[0,\epsilon]\cup[1-\epsilon,1]\}} to the constant map \s{t\mapsto I_n}. Hence there is a homotopy \s{P_s:I\ra \mathscr{P},s\in I} such that \s{P_0\equiv I_n}, and \s{P_1(t)=p(f)} and \s{P_s(t)\equiv I_n} when \s{t\in [0,\epsilon]\cup [1-\epsilon,1] }.   Suppose \s{P_s(t)=\begin{pmatrix}
     A_s(t)&B_s(t)\\
    0&1\\
 \end{pmatrix}}.
 Define a map \s{Q: I\times (D^{n-1}\times I)\ra \mr^{n-1}\times \mr } as follows:$$
 Q_s(y,t)=(A_s(t)y,t), s\in I, (y,t)\in D^{n-1}\times I
 $$

For each \s{s\in I}, \s{(DQ_s)_{(0,t)}=\begin{pmatrix}
     A_s(t)&0\\
    0&1\\
\end{pmatrix}}, hence \s{Q_s} is an immersion at \s{0\times I}. Moreover it fixes \s{0\times I}, and is the identity inclusion on a neighborhood of \s{D^{n-1}\times \{0,1\}}. We claim \s{Q_s} is an embedding at some neighborhood of \s{0\times I}: Otherwise we have two sequences \s{(y_m,t_m)_{m\in \mn},(y'_m, t'_m)_{m\in\mn}} in \s{D^{n-1}\times I} such that \s{y_m, y'_m} tends to \s{0} with \s{Q_s(y_m,t_m)=Q_s(y'_m,t'_m)} for all \s{m\geq 0}. Moreover by taking subsequences we may assume both \s{(y_m,t_m)} and \s{(y'_m,t'_m)} converge to some point in \s{0\times I}. Now the points they converge must be the same since the limit points have the same \s{Q_s}-images, but this contradicts to the local injectivity of \s{Q_s}.

For each \s{s\in I}, let \s{\epsilon(s)} be the supreme of the all radius \s{\epsilon} such that \s{Q_s} restricts to an embedding from \s{\epsilon D^{n-1}\times I} to \s{\mr^{n-1}\times I}, where \s{\epsilon D^{n-1}} is the standard \s{(n-1)}-dimensional ball with radius \s{\epsilon} centered at \s{0} in \s{\mr^{n-1}} (We will keep this notation). Observe that \s{\epsilon: I\ra \mr_{>0}} is lower semi-continuous, hence it has a minimum.
Thus we may pick a  \s{\delta>0} such that \s{Q_s:D^{n-1}_\delta\times I\ra S^{n-1}\times I} is an isotopy for all \s{s\in I}, and \s{Q_s(D^{n-1}_\delta\times I)}
is disjoint from \s{K\cup S^{n-1}\times \{0,1\}}. Now
we can extend \s{Q_{I\times (D^{n-1}_\delta\times I)}} to be an ambient isotopy (still denoted by \s{H}) on \s{S^{n-1}\times I} such that \s{supp(H)} is disjoint from \s{K\cup S^{n-1}\times \{0,1\}}. Observe that
 $D(H_3^{-1}f)_{(0,t)}=\begin{pmatrix}
    I_{n-1}&B_s(t)\\
    0&1\\
\end{pmatrix}$. By composing with \s{H^{-1}_3} we may assume \s{Df_{(0,t)}=\begin{pmatrix}
    I_{n-1}&B_s(t)\\
    0&1\\
\end{pmatrix}}.

Now define a map \s{J:I\times (D^{n-1}\times I)\ra \mr^{n-1}\times \mr} as follows: $$
J_s(y,t)=(1-s)(y,t)+sf(y,t),  s\in I, (y,t)\in D^{n-1}\times I
$$

We have \s{(DJ_s)_{(0,t)}=\begin{pmatrix}
    I_{n-1}&sB_1(t)\\
    0&1\\
\end{pmatrix}}, hence \s{J_s} is an immersion at \s{0\times I}, fixing \s{0\times I} and is the identical inclusion on a neighborhood of \s{D^{n-1}\times \{0,1\}}. Now analogously we can pick a neighborhood of \s{0\times I} such that \s{J_s} becomes an isotopy when restricted to this neighborhood.
And we can extend this restriction of \s{J} to an ambient isotopy on \s{S^{n-1}\times I} (still denoted by \s{J}), such that \s{J_s} is identity near \s{K\cup  S^{n-1}\times \{0,1\}}. 
Now let \s{\varphi:=J_1}. By construction, \s{\varphi} has the same germ with \s{f} on \s{0\times I} and \s{supp(\varphi)} is disjoint from \s{K\cup S^{n-1}\times \{0,1\}}, as required.

\end{proof}

The following statement, will be used to prove that the stabilizer of our group action is boundedly acyclic in \cref{stab ba}.

\begin{lemma}\label{iso of arc in Dn}
   Let \s{\alpha,\beta: I\ra \mr^{n-1}\times I} be two neat \s{C^r}-embeddings .   Let \s{K\subset \mr^{n-1}\times I}  be a set such that \begin{itemize}
        \item \s{K} is disjoint from \s{im(\alpha)} and \s{im(\beta)};
        \item  \s{K} is the union of  finitely many pair-wise disjoint \s{n}-dimensional smoothly embedded compact disks lying in the interior and finitely many  arcs in \s{\mr^{n-1}\times I}. 
    \end{itemize}
    
    Then there exists \s{\lambda \in \text{Diff}^r_0(\mr^{n-1}\times I)} with compact support such that \s{\lambda\alpha=\beta} and  \s{supp(\lambda)} is disjoint from \s{K}.
\end{lemma}
\begin{proof}
To be simple, in the proof we will assume all of the maps mentioned are \s{C^r}.
We  can assume that \s{\beta(t)=(0,t),t\in I}. We first can approximate \s{\alpha} by an neat embedding that is linear near the endpoints. Moreover, \s{\alpha} is ambient isotopic to such an neat embedding, with support disjoint from \s{K}. Hence we may assume \s{\alpha} has linear germs at endpoints.  Moreover locally rotating the vectors we may even assume the first derivative of \s{\alpha} is \s{(0,1)} when near the endpoints. Now \s{\alpha} is isotopic to a neat embedding \s{\gamma}, where$$\gamma(t)=(0,t)\in \mr^{n-1}\times I, t \text{ near the boundary }$$ Using the same technique in \cref{annuisoto}, we can perturb the isotopy so that it's image is disjoint from \s{K} and satisfies the condition of isotopy extension theorem \cref{iso-ex}. Hence \s{\alpha} is ambient isotopic to \s{\gamma} where the support of the ambient isotopy is disjoint from \s{K}. Now using the first part in the proof of \cref{annuisoto}, \s{\gamma} and \s{\beta} are ambient isotopic where its support is disjoint from \s{K}. The support of all of ambient isotopy mentioned can be made compact.

\end{proof}

\section{The main theorem}

In this section we assume that \s{n\geq 4} and \s{1\leq r\leq \infty} unless otherwise specified. So our goal in this section is to prove:

\begin{thm}\label{main diff}
 \s{\text{Diff}^r_+(S^n)} is boundedly acyclic for \s{n\geq 4,1\leq r\leq \infty}.
\end{thm}

We first need a term defined in \cite[Definition 3.1]{MN23}:

\begin{definition}\label{generic def}
Let $X$ be a set.  We call a binary relation $\perp$ \emph{generic} if  for every given finite set $Y \subseteq X$,  there is an element $x \in X$ such that $y \perp x$ for all $y \in Y$. 
\end{definition}
Recall that any generic relation on $X$ gives rise to a semi-simplicial set $X_{\bullet}^{\perp}$ in the following way:  we define $X_{n}^{\perp}$ to be the set of all $(n+1)$-tuples $\left(x_{0},  \ldots,  x_{n}\right) \in X^{n+1}$ for which $x_{i} \perp x_{j}$ holds for all $0 \leq j<i \leq n$.  The face maps $\partial_{n}:  X_{n}^{\perp} \rightarrow X_{n-1}^{\perp}$ are the usual simplex face maps,  \ie,

$$
\partial_{n}\left(x_{0},  \ldots,  x_{n}\right)=\sum_{i=0}^{n}(-1)^{i}\left(x_{0},  \ldots,  \widehat{x}_{i},  \ldots,  x_{n}\right),
$$

where $\widehat{x}_{i}$ means that $x_{i}$ is omitted.

\begin{lemma}\cite[Proposition 3.2]{MN23}\label{ge to ba}
    If the binary relation $\perp$ is generic, then the complex $X_{\bullet}^{\perp}$ is boundedly acyclic and connected.
\end{lemma}

To compute \s{H^\bullet_b\Big(\text{Diff}^r_+(S^n)\Big)} in low degrees \cite[Theorem 1.10]{MN23} , Monod and Nariman used the set of germs on \s{0} of all orientation-preserving \s{C^r}-embeddings \s{D^n\ra S^n}, denote this set by \s{X} and call every element of \s{X} a fat point, they called the image of \s{0} for each embedding the core of the fat point. Then they defined a generic relation \s{\perp} on \s{X} by setting for each \s{x,y\in X}, \s{x\perp y} if their cores are disjoint. Now \s{\text{Diff}^r_+(S^n)} acts highly transitively
on this complex, which means that the quotient complex is boundedly acyclic. (We say a group action is highly transitive if, for all \s{k\in \mn^+}, it is transitive on ordered k-tuples of distinct elements.)
Now we continue this idea to prove \s{\text{Diff}^r_+(S^n)} is boundedly acyclic. By \cref{gp to quotient}, We only need to verify all of the stabilizers of this action are boundedly acyclic.

We start by defining some subgroups of \s{\text{Diff}^r_+(S^n)}. Given a positive integer \s{k}, we fix some \s{k+1} pair-wise disjoint \s{n}-dimensional smoothly embedded compact disks in \s{S^{n-1}}, which is denoted by \s{B_0,\cdots, B_k}. Let \s{\Lambda_k= S^n-(\bigcup_{0\leq i\leq k}B_k)}
\begin{definition}
    Define \s{G_k} to be the group of elements \s{f\in \text{Diff}^r_+(S^n)} such that the germ of \s{f} at \s{\bigcup_{0\leq i\leq k}B_k}  is trivial.
\end{definition}
 Equivalently, \s{G_k} can be identified with \s{\text{Diff}^r_c(\Lambda_k)}.
Since the diffeomorphism type of \s{\text{Diff}^r_c(\Lambda_k)} is independent of the choice of \s{B_i}, the isomorphism type of \s{G_k} doesn't depend on the choice of these disks.  To be simple, we will always assume all of these \s{B_0,\cdots, B_k} are round disks if needed.

  For each \s{1\leq i\leq k}, we fix a smooth diffeomorphism \s{\alpha_i:S^{n}-int(B_0\cup B_i)\ra S^{n-1}\times I} such that \s{\alpha_i(\partial B_0)=S^{n-1}\times 0} and \s{\alpha_i(\partial B_i)=S^{n-1}\times 1}. Fix \s{k} disjoint \s{(n-1)}-dimensional round closed disks in \s{S^{n-1}} denoted by \s{C_1,\cdots, C_k}, such that \s{\alpha_i^{-1}(C_i\times I) } is disjoint from \s{\alpha_j^{-1}(C_j\times I) } for any \s{i\neq j}.

  \begin{definition}
    Given \s{\vec{y}=(y^1,\cdots,y^k)} and \s{\vec{y'}=(y'^1,\cdots,y'^k)} in \s{\prod_{1\leq i\leq k}C_i} we say they are disjoint if \s{y^i\neq y'^i} for all \s{1\leq i\leq k}, which is equivalent to say none of their coordinates coincide with each other. Since each \s{C_i} is an infinite set, this disjoint relation is a generic relation. Let \s{Z_\bullet} be the semi-simplicial complex consisting of pair-wise disjoint tuples. 
\end{definition}


Let \s{\vec{x}=(x^1,\cdots,x^k)\in \prod_{1\leq i\leq k} C_i}, then we define an equivalence relation \s{\sim_{\vec{x}}} on \s{G_k}: \s{g_1\sim_{\vec{x}} g_2 } when \s{g_1} and \s{g_2} have the same germ at  \s{\alpha_i^{-1}(x^i\times I) } for each \s{{1\leq i\leq k}}.

\textbf{Now we are ready to construct our semi-simplicial complex \s{Y_\bullet} as follows:  Define its \s{0}-skeleton \s{Y_0:=\bigsqcup_{\vec{x}\in \prod_{1\leq i\leq k} C_i}G_k/\sim_{\vec{x}}}.
We denote by \s{[g]_{\vec{x}}}  the equivalence class that \s{g} represents, here \s{g\in G_k, \vec{x}
\in\prod_{1\leq i\leq k} C_i }.
Given \s{[g_1]_{\vec{y_1}}} and \s{[g_2]_{\vec{y_2}}} in \s{Y_0},  we define a binary relation \s{\perp} on \s{Y_0} as follows: \s{[g_1]_{\vec{y_1}}\perp [g_2]_{\vec{y_2}}}  when \s{\bigcup_{1\leq i\leq k}g_1\alpha_i^{-1}(y^i_1\times I)} is disjoint from 
 \s{\bigcup_{1\leq i\leq k}g_2\alpha_i^{-1}(y^i_2\times I)}, where  we write \s{\vec{y_1}=(y^1_1,\cdots,y^k_1),\vec{y_2}=(y^1_2,\cdots,y^k_2)}.
Let \s{Y_\bullet} be the associated complex generated by \s{\perp}. Then \s{G_k} acts on \s{Y_0} in a natural way as $$h[g]_{\vec{x}} =[hg]_{\vec{x}}$$ for any \s{h\in G_k,[g]_{\vec{x}}\in Y_0}. Moreover, this action preserves the relation \s{\perp}, hence \s{Y_\bullet} is a  \s{G_k}-complex.
}
\begin{lemma}\label{generic}
     \s{\perp} is a generic relation.
\end{lemma}
\begin{proof}
   Let \s{\{[g_1]_{\vec{y_1}},\cdots,[g_m]_{\vec{y_m}}\}} be a finite subset  of $Y_0$. For each \s{1\leq i\leq m}, write \s{\vec{y_i}=(y^1_i,\cdots,y^k_i)}. We pick any \s{\vec{y}=(y^1,\cdots,y^k)\in \prod_{1\leq i\leq k} C_i} which is disjoint from \s{\vec{y_j}} for each \s{1\leq j\leq m}. By \cref{emb and trans} and \cref{diff op},  there exists \s{g\in \text{\text{Diff}}^r_+(S^n)} having the same germ with the identity map on \s{\bigcup_{0\leq i\leq k}B_i}, such that the restriction of \s{g} on \s{\bigcup_{1\leq i\leq k}\alpha_i^{-1}(y^i\times I)} is transverse to  \s{\bigcup_{1\leq i\leq k,1\leq j\leq m}g_j\alpha_i^{-1}(y^i_j\times I)}. Hence \s{\bigcup_{1\leq i\leq k}g\alpha_i^{-1}(y^i\times I)} and \s{\bigcup_{1\leq i\leq k,1\leq j\leq m}g_j\alpha_i^{-1}(y^i_j\times I)} are disjoint since they are both  \s{1}-dimensional and transversal in a manifold of dimension \s{\geq 4}. 
\end{proof}


\begin{lemma}\label{transitivity of orbit}
   Let \s{\sigma_p=([g_0]_{\vec{y_0}},\cdots,[g_p]_{\vec{y_p}}), \sigma'_p=([g'_0]_{\vec{y'_0}},\cdots,[g'_p]_{\vec{y'_p}})} be two \s{p}-simplicies in \s{Y_p}. Then they lie in the same orbit if and only if  \s{\vec{y_i}=\vec{y'_i}} for all \s{0\leq i\leq p}.
\end{lemma}
\begin{proof}
''Only if'' direction is clear, since \s{G_k} will fix the boundary.

For ''if'' direction, assume \s{\vec{y_i}=\vec{y'_i}} for all \s{0\leq i\leq p} and write \s{\vec{y_i}=\vec{y'_i}=(y^1_i,\cdots,y^k_i)}.
We induct on \s{p}, when \s{p=0} there is nothing to prove. Assume the lemma holds for \s{<p} cases, now we deal with the \s{=p} case. 
     By induction hypothesis there is \s{\kappa\in G_k} such that \s{\kappa([g_1]_{\vec{y_1}},\cdots,[g_p]_{\vec{y_p}})=([g'_1]_{\vec{y'_1}},\cdots,[g'_p]_{\vec{y'_p}})}.
      Hence it suffices to prove to the special case where \s{([g_1]_{\vec{y_1}},\cdots,[g_p]_{\vec{y_p}})=([g'_1]_{\vec{y'_1}},\cdots,[g'_p]_{\vec{y'_p}})}.

     
     Assume we have reduced the lemma to the special case where \s{([g_1]_{\vec{y_1}},\cdots,[g_p]_{\vec{y_p}})=([g'_1]_{\vec{y_1}},\cdots,[g'_p]_{\vec{y_p}})} and \s{g_0} has the same germ with \s{g'_0} on \s{\bigcup_{1\leq i< q}\alpha_i^{-1}(y_0^i\times I)} for some \s{1\leq q\leq k}.
     By definition, we have that \s{\bigcup_{1\leq i\leq k,i\neq q}B_i\cup\bigcup_{1\leq i< q}g_0\alpha_i^{-1}(y_0^i\times I)} is disjoint from \s{g_0\alpha^{-1}_q(y^q_0\times I)\cup g'_0\alpha^{-1}_q(y^q_0\times I)}.
     Then by the main lemma \cref{annuisoto} there is \s{\phi_q\in \text{Diff}^r_{c}(S^{n-1}\times I)} such that \s{\alpha_q^{-1}\phi_q\alpha_q} has the same germ with \s{g'_0g_0^{-1}} on \s{\alpha_q^{-1}(y_0^q\times I)}, and has the trivial germ at $$\bigcup_{0\leq i\leq k}B_i \cup\bigcup_{1\leq i\leq k,1\leq j\leq p}g_i\alpha_i^{-1}(y^i_j\times I) \cup\bigcup_{1\leq i< q}g_0\alpha_i^{-1}(y_0^i\times I)$$

     It follows that \s{\alpha_q^{-1}\phi_q\alpha_q\in G_k}.
     Hence by letting \s{\alpha_q^{-1}\phi_q\alpha_q} act on \s{\sigma_p}, we can  further reduce to the special case  where \s{([g_1]_{\vec{y_1}},\cdots,[g_p]_{\vec{y_p}})=([g'_1]_{\vec{y_1}},\cdots,[g'_p]_{\vec{y_p}})} and \s{g_0} has the same germ with \s{g'_0} on \s{\bigcup_{1\leq i< {q+1}}\alpha_i^{-1}(y_0^i\times I)}.
     Now iterate this inductive process, we finally reduce to the case when \s{\sigma_p=\sigma'_p}.


\end{proof}

\begin{lemma}\label{stab ba} For every \s{p\geq 0},
   there is only one isomorphism type of \s{p-}stabilizers of the \s{G_k}-action. And each \s{p-}stabilizer is boundedly acyclic.
\end{lemma}
\begin{proof}
The isomorphism type of  \s{\sigma_p=([g_0]_{\vec{y_0}},\cdots,[g_p]_{\vec{y_p}})} only depends on \s{(\vec{y_0},\cdots ,\vec{y_p})}. Therefore  by a simple fact that the natural \s{\text{Diff}^r(\overline{\Lambda_k})}-action on \s{\prod_{1\leq i\leq k}\partial B_i} is highly transitive, we deduce that each two \s{p}-stabilizers are conjugated by an element in  \s{\text{Diff}^r(\overline{\Lambda_k})} hence isomorphic to each other.

Next we prove they are boundedly acyclic, by \cref{transitivity of orbit} it suffice to prove the special case where \s{g_0=\cdots =g_p=id} and \s{(\vec{y_0},\cdots ,\vec{y_p})} is a pair-wise disjoint tuple. Now we can pick any balls \s{B_i} and diffeomorphisms \s{\alpha_i} as we like to proceed the proof.
First, pick a closed tubular neighborhood \s{U} of \s{\bigcup_{1\leq i\leq k}\alpha_i^{-1}(y^i_0\times I)} in \s{\overline{\Lambda_k}} such that there is a smooth diffeomorphism \s{\rho:\Omega:=(\overline{\Lambda_k}-int(U))\ra D^n}.
In fact, we can choose \s{B_i}  to  be some round disks in \s{S^n} and choose \s{\alpha_i} such that \s{\alpha_i^{-1}(y^i_0\times I)} are just  straight lines for all \s{i}.
Moreover, \s{U} can be chosen to be the union of some 'pipes' that connect \s{B_0} with \s{\bigcup_{1\leq i\leq k}B_i} in a smooth way, as illustrated in the first picture of \cref{figure 1}. Therefore,
\s{\bigcup_{0\leq i\leq k}B_i\cup U} is smoothly diffeomorphic to \s{D^n}. Hence  \s{\Omega}, the closure of its complement in \s{S^n}, is also smoothly diffeomorphic to \s{D^n}.

By our choice, we have that for each \s{1\leq i\leq k,1\leq j\leq m}, \s{\rho} sends  \s{\alpha_i^{-1}(y^i_j\times I)}  to a neat \s{C^r}-simple arc in \s{D^{n-1}}.  We can choose a homeomorphism \s{\rho': D^n\ra D^{n-1}\times I} such that 
\begin{itemize}
    \item \s{\rho'} is smooth except on \s{\rho'^{-1}(S^{n-1}\times \{0,1\})};
    \item For each \s{{1\leq i\leq k,1\leq j\leq p}}, \s{\rho'\rho\alpha_i^{-1}(y^i_j,l)\in int(D^{n-1})\times l} for \s{l\in \{0,1\}}.
\end{itemize}

Now \s{\rho'\rho \bigg(\bigcup_{1\leq i\leq k,1\leq j\leq p}\alpha_i^{-1}(y^i_j\times I)\bigg)} is the disjoint union of neatly embedded arcs in \s{int(D^{n-1})\times I}.
By \cref{iso of arc in Dn}, we can move each of these arcs, one by one while each time fixing all other arcs, to the standard neat embedding: \s{t\ra (x,t)} for some fixed \s{x\in int(D^{n-1})}.

Therefore the interior \s{\Omega':=int\bigg(\Omega-\bigcup_{1\leq i\leq k,1\leq j\leq p}\alpha_i^{-1}(y^i_j\times I)\bigg)} is diffeomorphic to \s{(\mr^{n-1}-P)\times \mr)} (see \cref{figure 1}), here \s{P} is a finite subset of \s{\mr^{n-1}} with \s{pk} points.
 Since \s{\text{Diff}^r_{c}((\mr^{n-1}-P)\times\mr)} is boundedly acyclic by \cref{ba for product},  we have \s{\text{Diff}^r_{c}(\Omega')} is boundedly acyclic.
 
Now $stab_{\sigma_p}(G_k)= \text{Diff}^r_c(\Lambda-\bigcup_{1\leq i\leq k,0\leq j\leq p}\alpha_i^{-1}(y^i_j\times I))$,
since \s{U} is a closed tubular neighborhood of \s{\bigcup_{1\leq i\leq k}\alpha_i^{-1}(y^i_0\times I)}, by \cref{deformcoame}, \s{stab_{\sigma_p}(G_k)} is coamenable to \s{\text{Diff}^r_{c}(\Omega')}.
Therefore \s{stab_{\sigma_p}(G_k)} is also boundedly acyclic.

   \begin{figure}
    \centering
    \includegraphics[width=0.9\linewidth]{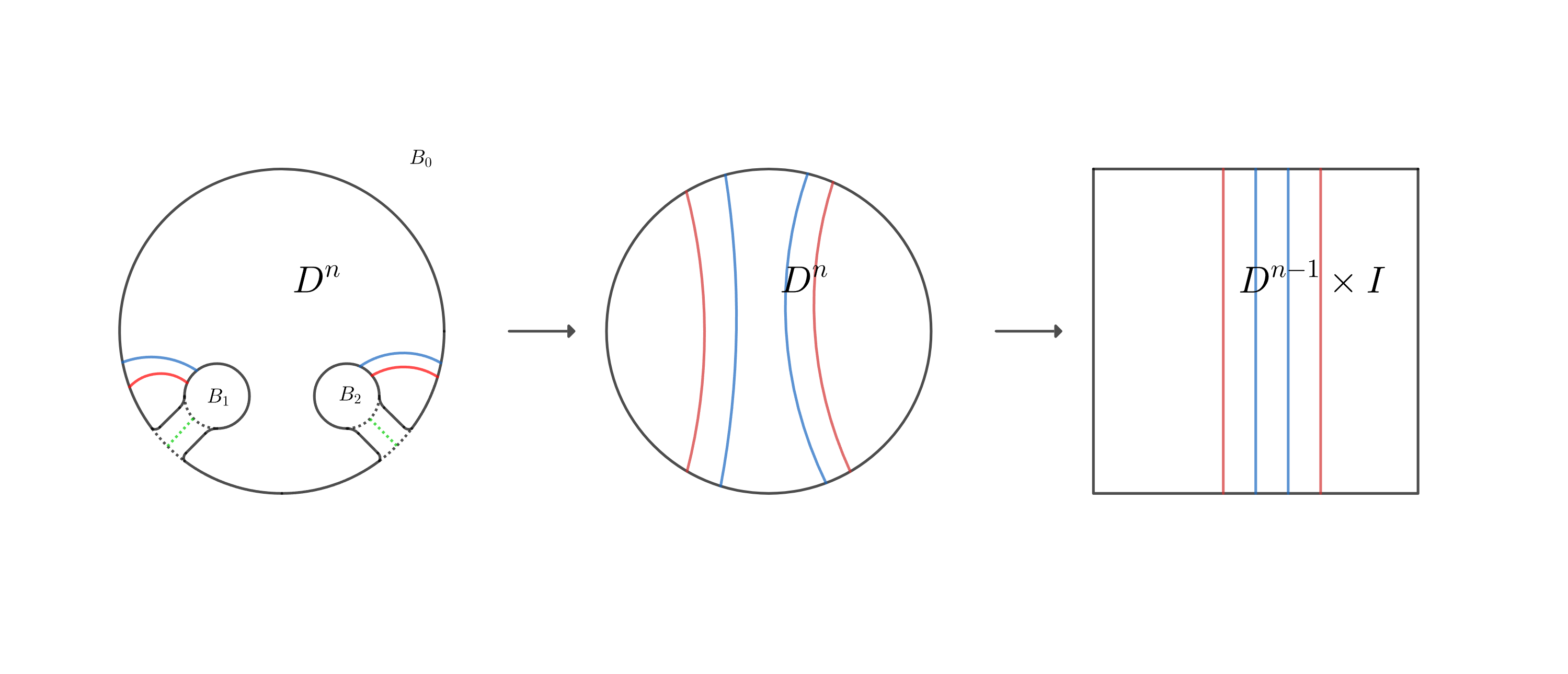}
    
    \caption{A way to see \s{V'\cong (\mr^{n-1}- P)\times\mr } when \s{k=2,p=2}, the blue and red lines are \s{\bigcup_{1\leq i\leq k,1\leq j\leq p}\alpha_i^{-1}(y^i_j\times I)}. }
    \label{figure 1}
\end{figure}
 
\end{proof}


\begin{lemma}\label{G_k ba}
   The quotient complex \s{Y_\bullet/G_k} is boundedly acyclic and connected, hence \s{G_k} is boundedly acyclic.
\end{lemma}
\begin{proof}

   There is a natural   map from \s{Y_\bullet/G_k } to \s{ Z_\bullet}, 
    given by \s{\sigma_p=([g_0]_{\vec{y_0}},\cdots,[g_p]_{\vec{y_p}})\mapsto (\vec{y_0},\cdots ,\vec{y_p})}.
   This is actually an isomorphism between complexes. Since \s{Z_\bullet} is connected and boundedly acyclic by \cref{ge to ba}, the first statement follows.

    Now by \cref{gp to quotient}, \s{G_k} is boundedly acyclic.
\end{proof}

\begin{proof}[Proof of \cref{main diff}]
    Consider the \s{\text{Diff}^r_+(S^n)}-action on the complex of fat points defined in \cite[Section 4]{MN23}. It suffices to prove the bounded acyclicity of all stabilizers. Each stabilizer is the subgroup of \s{\text{Diff}^r_+(S^n)} consisting of elements that has trivial germ at a given finite subset of \s{S^n}. By \cref{deformcoame} each stabilizer group is coamenable to \s{G_k} where \s{k} is the cardinality of this finite subset. By \cref{G_k ba}, \s{G_k} is boundedly acyclic.
\end{proof}

\section{Generalization}

In this chapter we prove:

\begin{thm}\label{3main}
    Let \s{V} be a smooth simply connected manifold of dimension \s{\geq 3}, such that \s{\partial V} is non-empty. Let \s{V_k} be the manifold  defined by deleting \s{k} pair-wise disjoint closed smoothly embedded \s{n}-dimensional disks in \s{int(V\times I)}. Then \s{\text{Diff}^r_c(V_k)} is boundedly acyclic.
\end{thm}

The tools and ideas to prove this theorem has already been prepared in this previous two chapters, so the proof should not differ too much from the previous proofs. 

Since \s{V} is simply connected, we can assume  it's oriented.
By the long exact sequence of fiberation : $$\pi_2(V\times I)\ra \pi_1(GL^{n}(\mr))\ra \pi_1(Fr^+(V\times I))\xrightarrow{p_*} \pi_1(V\times I)= 1$$
we have \s{\pi_1(Fr^+(V\times I^1)\leq \mz/2\mz}. Hence \s{Aut(\pi_1(Fr^+(V\times I)))} is trivial.
Since there is a canonical \s{\text{Diff}^r_c(V\times I)}-action on \s{Fr^+(V\times I)}, this action induces a trivial action on \s{\pi_1(Fr^+(V\times I^1))}.
Let \s{\Gamma=\text{Diff}^r_c(V\times I)}.

Let \s{\vec{e_1},\cdots \vec{e_{n-1}}} be the standard basis of \s{\mr^{n-1}}.
Let \s{\phi:\mr^{n-1}\ra U} be a diffeomorphism onto a open subset \s{U\subset V}. Define \s{D_i=2i\vec{e_1}+D^{n-1},1\leq i\leq k}, the translate of \s{D^{n-1}} by the vector \s{2i\vec{e_1}}. Let  \s{B_i=\phi(D_i)\times [\frac{1}{3},\frac{2}{3}]}.  It's not hard to show  that \s{int(V\times I)-\cup_{1\leq i\leq k}B_i} is diffeomorphic to \s{V_k}, and the diffeomorphism between them can be extended to a homeomorphism to the boundary.
Let \s{G= \text{Diff}^r_c(int(V\times I)-\cup_{1\leq i\leq k}B_i)}. So to prove \cref{3main} is suffices to prove \s{G} is boundedly acyclic.

For \s{u\in \bigcup_{1\leq i\leq k}\phi(D_i)}, define \s{\Gamma_u:I\ra Fr^+(V\times I)} as \s{\Gamma_u(t)=[(u,\frac{t}{3}),(\phi_*(\vec{e_1}),\cdots, \phi_*(\vec{e_{n-1}}),\frac{d}{dt})]} where \s{u\in \phi(D_i)}. Let \s{\gamma_u: t\mapsto (u,\frac{t}{3}), t\in I}. For \s{f\in G}, let \s{f_*} be the induced map on \s{Fr^+(int(V\times I)-\cup_{1\leq i\leq k}B_i)}. Then \s{f_*\Gamma_u} is homotopic to \s{\Gamma_u} rel the endpoints \s{\{0,1\}}, since \s{G} acts trivially on \s{\pi_1(Fr^+(int(V\times I)-\cup_{1\leq i\leq k}B_i))\cong \pi_1(Fr^+(V\times I))}.

For \s{x\in \prod:=\prod_{1\leq i\leq k} \phi(int(D_i))} we always write its \s{i}-th coordinate as \s{x^i}, \ie, \s{x=(x^1,\cdots,x^k)}.  Similar to \ref{}, we define \s{x,y\in \prod} are disjoint, if the set of coordinates of \s{x} is disjoint form the set of coordinates of \s{y}. Let \s{\prod_\bullet} be the complex associated to this disjoint relation, \ie, the \s{\prod_p} is the set of \s{p+1} pair-wise disjoint tuples in \s{X^{p+1}}, with the natural face relations.

\begin{thm}\label{3.2 also main}
    \s{G} is boundedly acyclic. Hence \s{\text{Diff}^r_c(V_k)} is boundedly acyclic.
\end{thm}

\begin{proof}

    Given \s{x\in\prod},
    let \s{G_{x}} be the group of elements of \s{G} that have the trivial germ at \s{\bigsqcup_{1\leq i\leq k} im(\gamma_{x^i})}.
    Let \s{X_0:=\bigsqcup_{x\in \prod} G/G_x}.
    Define an binary relation on \s{X_0} as follows: For \s{g_1G_{x_1}, g_2G_{x_2}\in X_0}, write \s{x_1=(x_1^1,\cdots x_1^k)}, \s{x_2=(x_2^1,\cdots,x_2^k)}. Define \s{g_1G_{x_1}\perp g_2G_{x_2} } if \s{g_1(\bigcup_{1\leq i\leq k}im(\gamma_{x^i_1}))}  and \s{ g_2(\bigcup_{1\leq i\leq k}im(\gamma_{x^i_2}))} are disjoint.
A necessary condition for this relation is that the set of all coordinates of \s{x_1} is disjoint from the set of all coordinates of \s{x_2}, in this case we say \s{x_1} is disjoint from \s{x_2}.

    Let \s{X_\bullet} be the associated complex defined by this relation. That is, every \s{m}-simplex in \s{X_m} is a pair-wise \s{\perp}-tuple \s{(g_0G_{x_0},\cdots g_mG_{x_m})}. Now \s{G} acts simplicially on \s{X_\bullet}.

    First we prove \s{\perp} is generic:  For \s{\{g_1G_{x_1},\cdots g_mG_{x_m}\}\subset X_0}, there is \s{y=(y^1,\cdots ,y^k)\in \prod} that is disjoint from each \s{x_j,1\leq j\leq m}. And by transversality there is a \s{g\in G} which is a small perturbation of  \s{id\in G}, such that for each \s{1\leq i\leq k}, \s{g\gamma_{y^i}} is transversal hence  disjoint from  \s{\bigcup_{1\leq i\leq k, 1\leq j\leq m}im(g_j\gamma_{x_j^i})} . So \s{gG_{y}\perp g_jG_{x_j}} for each \s{1\leq j\leq m} . This shows \s{\perp} is generic.

    Next we prove the quotient complex is isomorphic to \s{\prod_\bullet}, which is boundedly acyclic. 
    Define chain map \s{F_\bullet: X_\bullet/G\ra \prod_\bullet} as follows:
    For  \s{\sigma=G(g_0G_{x_0},\cdots g_mG_{x_m})\in X_m/G}, define \s{F_m(\sigma)=(x_0,\cdots ,x_m)}. It's a well defined chain map. It's surjective since for every \s{(x_0,\cdots ,x_m)\in \prod_m}, \s{F_m(G(idG_{x_0},\cdots idG_{x_m}))=(x_0,\cdots ,x_m)}. It's injective by \cref{3.5trans}.

    Finally we prove every stabilizer of the \s{G}-action on \s{X_\bullet} is uniformly boundedly acyclic. By \cref{3.5trans}, the stabilizer of \s{(g_0G_{x_0},\cdots g_mG_{x_m})} is isomorphic to the stabilizer of \s{(idG_{x_0},\cdots idG_{x_m})}, which is \s{  \text{Diff}^r_c(V_k-\bigcup_{1\leq i\leq k,0\leq j\leq m} im(\gamma_{x_j^i}))}. Since any compact subset of \s{V_k-\bigcup_{1\leq i\leq k,0\leq j\leq m} im(\gamma_{x_j^i})} can be deformed into \s{int(  N_{k,m})-\bigcup_{1\leq i\leq k,1\leq j\leq m} im(\gamma_{x_j^i})} via a diffeomorphism in \s{G},  we see that this stabilizer is coamenable to \s{\text{Diff}^r_c(int(  N_{k,m})-\bigcup_{1\leq i\leq k,1\leq j\leq m} im(\gamma_{x_j^i}))}. Here \s{N_{k,m}} is defined in \cref{3.5trans}. Using \cref{3.3 ambient iso} multiple times by setting \s{M} to be \s{(V-\text{ finitely many points })}, we see \s{int(  N_{k,m})-\bigcup_{1\leq i\leq k,1\leq j\leq m} im(\gamma_{x_j^i})} is diffeomorphic to \s{(int(V-P_{k,m}))\times (0,1)}, where \s{P_{k,m}} is a finite set in \s{int(V)} with \s{km} points. For an intuition, see \cref{figure 2}.    
    Since by \cref{ba for product}, \s{\text{Diff}^r_c((int(V-P_{k,m}))\times (0,1))} is boundedly acyclic, the result follows.
\end{proof}

\begin{figure}\label{V_times I}
    \centering
    \includegraphics[width=0.9\linewidth]{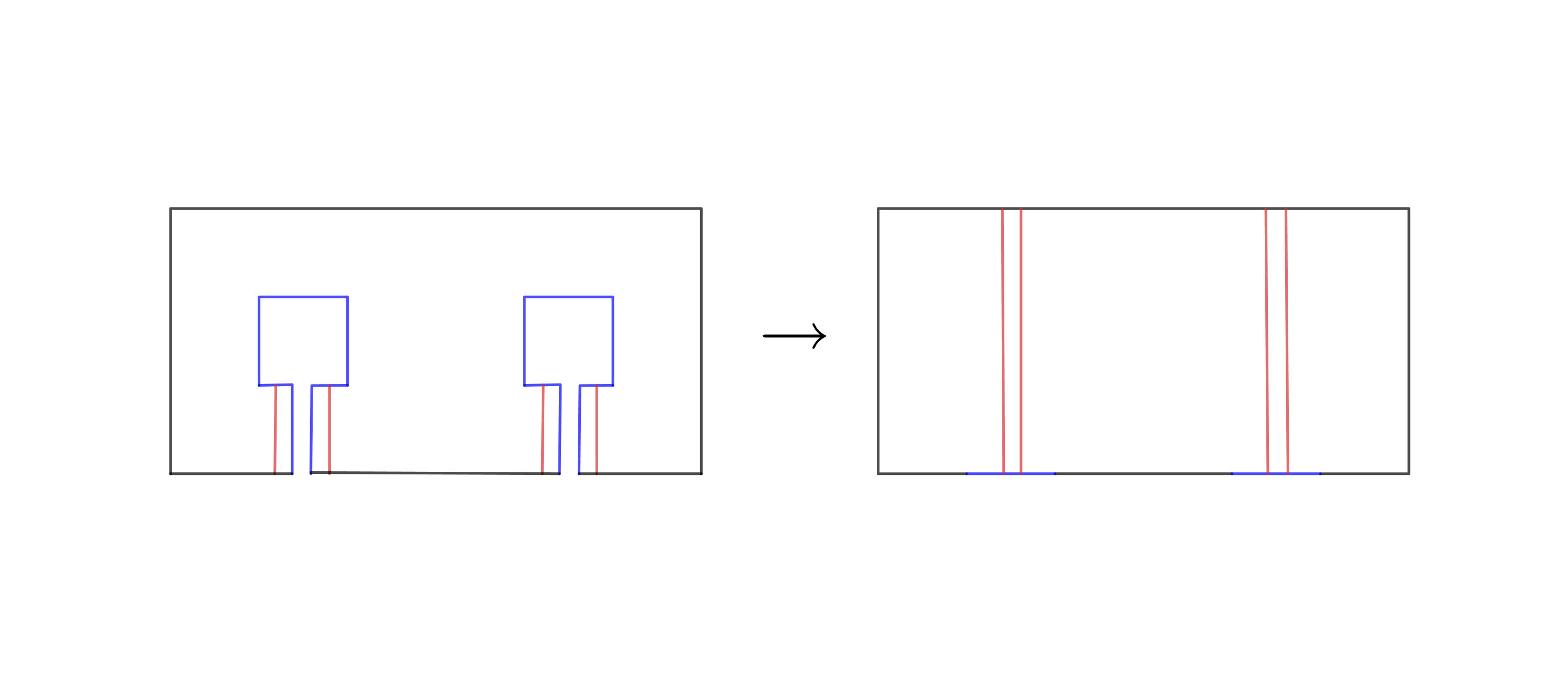}
    \caption{A way to see \s{int(  N_{k,m})-\bigcup_{1\leq i\leq k,1\leq j\leq m} im(\gamma_{x_j^i})} is diffeomorphic to \s{(int(V-P_{k,m}))\times (0,1)} when \s{k=2,m=2}.  The red lines are \s{\bigcup_{1\leq i\leq k,1\leq j\leq m}im(\gamma_{x_j^i})}. }
    \label{figure 2}
\end{figure}

\begin{lemma}\label{3.3 ambient iso}Let \s{M} be a smooth simply connected manifold of dimension \s{\geq 3}, such that \s{\partial M\neq \varnothing}.
    Then every two neat \s{C^r}-embeddings of \s{I} into \s{M\times I} are  ambient \s{C^r}-isotopic. Moreover, if these two embedding have the same germ at the endpoints, then the isotopy can be chosen to have compact support in the interior of \s{M\times I}.
\end{lemma}
\begin{proof}
    We first prove they are \s{C^r}-isotopic, then use isotopy extension theorem to get the ambient isotopy. 

   Let \s{f,g: I\ra M\times I} be two neat embeddings.  We may assume \s{f(t)=(v,t),t\in I} where \s{v\in M} is a given point. Since \s{\partial (V\times I)} is connected, by \cref{3.4rmk} we can modify \s{g} by composing a \s{H\in \text{Diff}^r(int(M\times I))\cap Homeo_0(M\times I)} so that \s{f,Hg} have the same image at the endpoints. Moreover, similar to \cref{iso of arc in Dn}, by composing  an appropriate diffeomorphism \s{H'} which is isotopic to the identity such that \s{supp(H')} is contained in a small neighborhood of \s{Hg(\{0,1\})}, we have \s{H'Hg} has the same germ with \s{f} at the endpoints. So now we assume that \s{g} has the same germ with \s{f} at endpoints. Now we follow the first two steps of the proof of \cref{annuisoto}. We can approximate \s{\pi_1g} by a map \s{c: I\ra M} such that \s{c\equiv v} on \s{[0,\epsilon]\cup [1-\epsilon,1]} and \s{c_{[\epsilon,1-\epsilon]}} is an embedding. Then \s{g':t\mapsto (c(t),\pi_2g(t)} is also an embedding when \s{c} is \s{C^1}-close enough to \s{\pi_1g}.
    We have \s{g'} is isotopic to \s{h:t\ra (c(t),t)}.  
 Let \s{c_s:I\ra M,s\in I} be a smooth homotopy such that \s{c_0=c,c_1\equiv v,(c_s)_{[0,\epsilon]\cup [1-\epsilon,1]}\equiv v, \forall s\in I}.
 Then \s{h} is isotopic to \s{h':t\mapsto (v,t)} by \s{h_s:t\mapsto (c_s(t),t)}.
 
 Now we have a isotopy \s{H} from \s{f} to \s{g} which coincides with \s{f} at a neighborhood of endpoints. Hence we can extend it to be a ambient isotopy by \cref{iso-ex}, where the support lies in a small neighborhood of \s{im(H)}.
\end{proof}

\begin{remark}\label{3.4rmk}
    Recall a well-known fact that given a smooth manifold \s{M=M^m}, 
    if two points \s{a,b} lie in the same component of  \s{\partial M}, then there exist \s{H\in \text{Diff}^\infty_0(M)} which sends \s{a} to \s{b}. We remark that the similar result also holds, when \s{M} is a differential manifold with corners, \ie, it's locally diffeomorphic to open sets of \s{\mr_+^m}. at this time \s{H\in \text{Homeo}_0(M)\cap \text{Diff}^\infty(int(M))}. We only need to specify the case when \s{a,b} are near the corner. To do this, we only need to use the fact there is a homeomorphism between \s{\mr^{n-1}\times \mr_+} and \s{\mr^{n-k}\times \mr_+^{k}} which is differentiable at points outside the corner for \s{1\leq k\leq n}.
\end{remark}

\begin{lemma}\label{3.5trans}
    \s{G(g_0G_{x_0},\cdots g_mG_{x_m})=G(g'_0G_{x'_0},\cdots g'_mG_{x'_m})} if and only if \s{(x_0,\cdots ,x_m)=(x'_0,\cdots ,x'_m)}.
\end{lemma}
\begin{proof}
    The "only if"  direction is clear. Let's prove the "if" direction. Assume \s{(x_0,\cdots ,x_m)=(x'_0,\cdots ,x'_m)}. Then \s{g'_0g_0^{-1}g_0G_{x_0}=g'_0G_{x'_0}}. Without the loss of generality we assume \s{g_0=g'_0=id}. By definition, for each \s{1\leq j\leq m}, \s{g_j(\bigcup_{1\leq i\leq k}im(\gamma_{x^i_j}))} and \s{g'_j(\bigcup_{1\leq i\leq k}im(\gamma_{x^i_1}))} are disjoint from \s{\bigcup_{1\leq i\leq k}im(\gamma_{x^i_0})}. Hence \s{\bigcup_{1\leq j\leq m}(g_j(\bigcup_{1\leq i\leq k}im(\gamma_{x^i_j})) \cup g'_j(\bigcup_{1\leq i\leq k}im(\gamma_{x^i_j})))} lies in \s{int(V\times I)-\bigcup_{1\leq i\leq k}\phi(x^i_0+\epsilon D^{n-1})\times [0,\frac{1}{3}]} for some small enough \s{\epsilon>0}, where \s{\epsilon D^{n-1}} is the standard open ball in \s{\mr^{n-1}} with radius \s{\epsilon} centered at 0, and \s{x+ K:=\{x+k:k\in K\}}. 
    
Let \s{  N_{k,m}:=V\times I-\bigcup_{1\leq i\leq k}\phi(x_0^i+ int(\epsilon D^{n-1}))\times (0.\frac{1}{3})-\bigcup_{1\leq i\leq k}\phi(int(D_i))\times (\frac{1}{3},\frac{2}{3})}.
    Note that there is a homeomorphism $$\rho:   N_{k,m}\ra V\times I$$ where \s{\rho} restricts to a diffeomorphism outside the corner \s{\partial V\times \{0,1\}} of the boundary. Conjugated by \s{\rho}, each pair \s{(g_j\gamma_{x_j^i}, g_j\gamma_{x_j^i})} becomes a pair of neatly embedded \s{C^r}-arcs in \s{V\times I}, such that \s{g_j\gamma_{x_j^i}, g_j\gamma_{x_j^i}} have the same germs at the endpoints. 
    
    Now \s{g_j\gamma_{x_j^i}} and \s{g'_j\gamma_{x_j^i}} can be identified as neat embeddings of \s{I} into \s{V\times I}.  We  claim that:  There is an \s{h\in \text{Diff}^r_c(V\times I)}, such that \s{hg_j\gamma_{x_j^i}=g'_j\gamma_{x_j^i}} for each \s{1\leq i\leq k, 1\leq j\leq m}. This can be done by taking \s{M=(V- \text{ finitely many points })} in \cref{3.3 ambient iso}. At each time we only move an arc while fixing all other arcs, eventually we get the element \s{h}.
    So without the loss of generality we can moreover assume \s{g_j\gamma_{x_j^i}=g'_j\gamma_{x_j^i}} for each \s{1\leq j\leq m,1\leq i\leq k}.

    Pick a contractible neighborhood \s{U} of \s{im(\gamma_{x_j^i})}, Then the composition $$I \xrightarrow{(g'_jg_j^{-1})_*\Gamma_{x_j}} Fr^+(U)=U\times GL^+_n(\mr) \longrightarrow
 GL^+_n(\mr)$$
 is homotopic rel \s{\{0,1\}} to \s{\Gamma_{x_j}} as a path in \s{GL^+_n(\mr)}.
Now following the third step of the  proof of \cref{annuisoto},  We can find a \s{h'\in \text{Diff}_0(  N_{k,m})\cap \text{Diff}^r_c(int(  N_{k,m}))} such that \s{h'g_j} has the same germ with \s{g'_j} at \s{\bigcup_{1\leq i\leq k}im(\gamma_{x_j^i})} for each \s{1\leq j\leq m}. Hence \s{h'(g_0G_{x_0},\cdots g_mG_{x_m})=(g'_0G_{x'_0},\cdots g'_mG_{x'_m})}, as required.   
\end{proof}

\bibliographystyle{alpha}
\bibliography{references.bib}

\begin{thebibliography}{FFMNK24}

\bibitem[BHW22]{BHW22}
Jonathan Bowden, Sebastian~Wolfgang Hensel, and Richard Webb.
\newblock Quasi-morphisms on surface diffeomorphism groups.
\newblock {\em J. Amer. Math. Soc.}, 35(1):211--231, 2022.

\bibitem[CFFLM23]{CFFLM23}
Caterina Campagnolo, Francesco Fournier-Facio, Yash Lodha, and Marco Moraschini.
\newblock An algebraic criterion for the vanishing of bounded cohomology.
\newblock {\em arxiv:2311.16259}, 2023.

\bibitem[FFMNK24]{FNS24}
Francesco Fournier-Facio, Nicolas Monod, Sam Nariman, and Alexander Kupers.
\newblock The bounded cohomology of transformation groups of euclidean spaces and discs.
\newblock {\em arxiv:2405.20395}, 2024.

\bibitem[Hir76]{Hir76}
Morris~W. Hirsch.
\newblock {\em Differential topology}, volume No. 33 of {\em Graduate Texts in Mathematics}.
\newblock Springer-Verlag, New York-Heidelberg, 1976.

\bibitem[MN23]{MN23}
Nicolas Monod and Sam Nariman.
\newblock Bounded and unbounded cohomology of homeomorphism and diffeomorphism groups.
\newblock {\em Invent. Math.}, 232(3):1439--1475, 2023.

\bibitem[Mon22]{Mon22}
Nicolas Monod.
\newblock Lamplighters and the bounded cohomology of {T}hompson's group.
\newblock {\em Geom. Funct. Anal.}, 32(3):662--675, 2022.

\end{thebibliography}
  
\end{document}